\documentclass[12pt,reqno]{amsart}
\usepackage{amssymb}
\usepackage{amsmath}
\usepackage{enumerate}
\usepackage{mathrsfs}
\usepackage{float}
\usepackage{colonequals}
\usepackage[all]{xy}
\usepackage[table]{xcolor}
\usepackage{comment}
\usepackage{marginnote}
\usepackage[english]{babel}
\usepackage{fullpage}
\usepackage{mathtools}
\usepackage{tikz-cd}

\usepackage{pdflscape}
\usepackage{booktabs}
\usepackage{longtable}

\usepackage{wasysym}

\usepackage{tikz}

\usepackage{hyperref}

\hypersetup{pdftitle={Computing Euclidean Belyi maps},pdfauthor={Matthew Radosevich, John Voight}}
\hypersetup{colorlinks=true,urlcolor=blue,linkcolor=blue,anchorcolor=blue,citecolor=blue}

\RequirePackage{mathrsfs} %\let\mathcal\mathcal

\numberwithin{equation}{subsection} 

\theoremstyle{plain}
\newtheorem{theorem}[equation]{Theorem}
\newtheorem{thm}[equation]{Theorem}
\newtheorem{lemma}[equation]{Lemma}

\newtheorem{prop}[equation]{Proposition}
\newtheorem{corollary}[equation]{Corollary}
\newtheorem{corr}[equation]{Corollary}

\newtheorem{alg}[equation]{Algorithm}

\theoremstyle{definition} 

\newtheorem{defn}[equation]{Definition}
\newtheorem{algorithm}[equation]{Algorithm}

\theoremstyle{remark}
\newtheorem{remark}[equation]{Remark}
\newtheorem{example}[equation]{Example}

\newenvironment{enumalph}
{\begin{enumerate}

}
{\end{enumerate}}

\newenvironment{enumalg}
{\begin{enumerate}}
{\end{enumerate}}

\DeclareMathOperator{\Gal}{Gal}

\newcommand{\Qbar}{\mathbb{Q}^{\textup{al}}}

\newcommand{\Q}{\QQ}

\newcommand\PP{\mathbb{P}}

\newcommand\QQ{\mathbb{Q}}

\newcommand\C{\mathbb{C}}

\newcommand\Z{\mathbb{Z}}

\newcommand{\PoneC}{\mathbb{P}^1_\C}

\newcommand{\calH}{{\mathcal H}}

\newcommand{\defi}[1]{\textsf{#1}} 	% for defined terms

\tikzstyle{blackdot}=[fill=black, draw=black, shape=circle]
\tikzstyle{new style 0}=[fill={rgb,255: red,74; green,74; blue,74}, draw={rgb,255: red,85; green,85; blue,85}, shape=circle]
\tikzstyle{light grey dot}=[fill={rgb,255: red,185; green,185; blue,185}, draw={rgb,255: red,194; green,194; blue,194}, shape=circle]
\tikzstyle{Red Dot}=[fill=red, draw=red, shape=circle]
\tikzstyle{Blue Dot}=[fill=blue, draw=blue, shape=circle]
\tikzstyle{green dot}=[fill={rgb,255: red,28; green,195; blue,16}, draw={rgb,255: red,90; green,162; blue,68}, shape=circle]
\tikzstyle{new style 1}=[fill={rgb,255: red,128; green,0; blue,128}, draw={rgb,255: red,128; green,0; blue,128}, shape=circle]
\tikzstyle{new style 2}=[fill=cyan, draw=cyan, shape=circle]
\tikzstyle{new style 3}=[fill={rgb,255: red,255; green,128; blue,0}, draw={rgb,255: red,255; green,128; blue,0}, shape=circle]
\tikzstyle{new style 4}=[fill=green, draw=green, shape=circle]
\tikzstyle{new style 5}=[fill={rgb,255: red,128; green,128; blue,0}, draw={rgb,255: red,128; green,128; blue,0}, shape=circle]
\tikzstyle{new style 6}=[fill={rgb,255: red,66; green,154; blue,69}, draw={rgb,255: red,66; green,154; blue,69}, shape=circle]
\tikzstyle{open dot}=[fill=white, draw=black, shape=circle]

\tikzstyle{new edge style 0}=[-, draw={rgb,255: red,255; green,128; blue,0}]
\tikzstyle{black}=[-]
\tikzstyle{R6}=[-, draw=blue]
\tikzstyle{R4}=[-, draw=cyan]
\tikzstyle{R3}=[-, draw=green]
\tikzstyle{R2}=[-, draw={rgb,255: red,255; green,128; blue,0}]
\tikzstyle{c4}=[-, draw={rgb,255: red,128; green,128; blue,0}]
\tikzstyle{c3}=[-, draw={rgb,255: red,5; green,121; blue,9}]
\tikzstyle{R1}=[-, draw=red]
\tikzstyle{c6}=[-, draw={rgb,255: red,128; green,0; blue,128}]
\tikzstyle{new edge style 1}=[->]
\tikzstyle{blwitharrowotherway}=[<-]
\tikzstyle{new edge style 2}=[-, draw={rgb,255: red,207; green,207; blue,207}]
\tikzstyle{lightblack}=[-, draw={rgb,255: red,50; green,50; blue,50}]

\title{Computing Euclidean Belyi maps}

\author{Matthew Radosevich}
\address{Department of Mathematics,
  Dartmouth College, 6188 Kemeny Hall, Hanover, NH 03755, USA}
\email{matt.j.radosevich@gmail.com}

\author{John Voight}
\address{Department of Mathematics,
  Dartmouth College, 6188 Kemeny Hall, Hanover, NH 03755, USA}
\email{jvoight@gmail.com}

\subjclass[2010]{11G32, 11Y40}
\date{\today}

\begin{document}

\begin{abstract}
We exhibit an explicit algorithm to compute three-point branched covers of the complex projective line when the uniformizing triangle group is Euclidean.  
\end{abstract}
 
\maketitle

\setcounter{tocdepth}{1}
% \tableofcontents

\section{Introduction}

\subsection{Motivation}

Grothendieck in his \emph{Esquisse d'un Programme} \cite{Grothendieck} described an action of the absolute Galois group $\Gal(\Qbar\,|\,\Q)$ of the rational numbers on the sets of Belyi maps and dessins d'enfants, linking combinatorics, topology, geometry, and arithmetic in a deep and surprising way.  Computational aspects of this program \cite{SijslingVoight} remain of significant interest, and there has been recent, fundamental progress using complex analytic techniques \cite{KMSV,Monien,Co3,BarthWenz}.  A common thread underlying these approaches is to realize a Belyi map via uniformization as $\varphi \colon \Gamma \backslash \calH \to \Delta \backslash \calH$ where $\calH$ is one of the three classical geometries (the sphere, the Euclidean plane, or the hyperbolic plane), and $\Gamma \leq \Delta$ is a finite-index subgroup of a triangle group.  The case where $\calH$ is spherical is truly classical, corresponding to certain triangulations of the Platonic solids.  In the hyperbolic case, analytic methods can be employed to convert this geometric description into an algebraic one, using  modular forms, finite element techniques, or conformal maps.  What remains is the case of Euclidean triangle groups, those arising from the familiar regular triangular tessellations of the Euclidean plane.  In this paper, we fill this gap: we compute Euclidean Belyi maps explicitly from maps of complex tori, forming a bridge between the classical and the general.

\subsection{Main result}

A \defi{Belyi map} over $\C$ is a morphism $\varphi \colon X \to \PP^1_\C$ of nice (projective, nonsingular, integral) curves over $\C$ that is unramified away from $\{0,1,\infty\}$.  By the Riemann existence theorem, we may equivalently work with such a map of compact Riemann surfaces.  Famously, Belyi \cite{Belyi,Belyi2} proved that a curve $X$ over $\C$ can be defined over the algebraic numbers $\Qbar$ if and only if $X$ admits a Belyi map.  

Belyi maps admit a tidy combinatorial description, something we take as the input to our algorithm.  A \defi{permutation triple} of degree $d$ is a triple $(\sigma_0,\sigma_1,\sigma_\infty) \in S_d^3$ of permutations on $d$ elements such that $\sigma_\infty \sigma_1 \sigma_0 = 1$.  A permutation triple is \defi{transitive} if it generates a transitive subgroup of $S_d$. 
The monodromy around $0,1,\infty$ of a Belyi map of degree $d$ gives a permutation triple of degree $d$, giving a bijection between isomorphism classes of Belyi maps of degree $d$ and transitive permutation triples up to simultaneous conjugation.  Lifting paths, one can compute (by numerical approximation) the permutation triple attached to a Belyi map; in this paper, we consider the harder, converse computational task.

Let $\sigma$ be a transitive permutation triple of degree $d$ and let $a,b,c$ be the orders of $\sigma_0,\sigma_1,\sigma_\infty$, respectively.  By the theory of covering spaces, the permutation triple $\sigma$ defines a homomorphism $\pi \colon \Delta(a,b,c) \to S_d$ and thereby a subgroup $\Gamma \leq \Delta(a,b,c)$ of index $d$ (see section \ref{sec:background}).  The quotient $\Gamma \backslash \calH$ can be given the natural structure of a Riemann surface $X(\Gamma)$, and the further quotient to $\Delta \backslash \calH$ defines a Belyi map $\varphi \colon X(\Gamma) \to X(\Delta) \simeq \PP^1_\C$.  By the theorem of Belyi, the map $\varphi$ can be defined over the field of algebraic numbers $\Qbar$.  

We say that $\sigma$ (and its corresponding map $\varphi$) is \defi{Euclidean} if $1/a+1/b+1/c=1$, in which case the attached triangle group $\Delta(a,b,c)$ is a group of symmetries of the Euclidean plane, whence $(a,b,c)=(3,3,3),(2,3,6),(2,4,4)$.  Our main result provides an algorithmic way to compute algebraic equations for $\varphi$ given $\sigma$.

\begin{thm} \label{thm:mainthm}
There exists an explicit algorithm that, given as input a transitive, Euclidean permutation triple $\sigma$, produces as output a model for the Belyi map $\varphi$ associated to $\sigma$ over $\Qbar$.  
\end{thm}

The algorithm in Theorem \ref{thm:mainthm} is specified in Algorithm \ref{alg:thisisit}.  We implemented the algorithm in the computer algebra system \textsf{Magma} \cite{Magma}: the running time is quite favorable.  We computed a database of Euclidean Belyi maps with this implementation (see section \ref{sec: Examples}) which we will upload to the LMFDB \cite{LMFDB}.  Our code is available as part of a Belyi maps package available online (\url{https://github.com/michaelmusty/Belyi}).

\begin{remark}
It would be interesting to estimate the running time of our algorithm by estimating the heights of intermediate computations and the precision required in Step 4 of Algorithm \ref{alg: isogeny}.
\end{remark}

\subsection{Proof sketch}

We now briefly indicate the idea behind the proof of Theorem \ref{thm:mainthm}.  We first convert the permutation triple $\sigma$ into an explicit description of the group $\Gamma \leq \Delta$.  Next, we write $\Gamma \simeq T(\Gamma) \rtimes R(\Gamma)$  as a semi-direct product, where $T(\Gamma)$ consists of the subgroup of translations in $\Gamma$ and $R(\Gamma)$ is generated by rotation around a particular point, which we can find explicitly.  The quotients $E(\Gamma) \colonequals T(\Gamma) \backslash \C$ and $E(\Delta) \colonequals  T(\Delta) \backslash \C$ define elliptic curves.  We then have the following commutative diagram, which we call the \emph{master diagram}:
\begin{equation}  \label{eqn:EXDeltaGamma}
\begin{gathered}
\xymatrix@R=1ex@C=6ex{
E(\Gamma) \ar[rd]^{\beta} \ar[dd]_{\psi} & \\
& X(\Gamma) \ar[dd]^{\varphi} \\
E(\Delta) \ar[rd]_(.4){\alpha} \\
& X(\Delta) \simeq \PP^1
} 
\end{gathered}
\end{equation}

% $C/\Gamma$ gives either an elliptic curve or a copy of $\PoneC$, so we will call the general surface $X(\Gamma)$. 
% $\C/\Delta$ gives a copy of $\PoneC$ that we call $X(\Delta)$. 
To find the Belyi map $\varphi$, our strategy is to compute the other three maps in our diagram, filling in $\varphi$ by commutativity (``descending $\psi$ along $\alpha$'').  The bottom map $\alpha$ depends only on $a,b,c$ and the choice of origin, giving six possibilities.  The map $\psi$ is an isogeny of elliptic curves, which we compute from the inclusion of lattices implied by $T(\Gamma) \leq T(\Delta)$ by applying formulas of V\'elu. The top map $\beta$ is computed by looking at the fixed field of $\C(E(\Gamma))$ under the finite subgroup of automorphisms corresponding to the rotations $R(\Gamma)$ (taking care to ensure these rotations act by automorphisms at the origin).  The final step, to fill in $\varphi$ to make the diagram commute, is obtained via explicit substitution.  

\subsection{Contents}

After reviewing background in section \ref{sec:background}, we exhibit in section \ref{sec:mainresult} the main algorithm (Algorithm \ref{alg:thisisit}) in pseudocode and then prove the main result (Theorem \ref{thm:mainthm}).  In section \ref{sec: Examples} we describe an implementation in the \textsf{Magma} computer algebra system and then present some computed examples.

\subsection{Acknowledgements}

The authors would like to thank Sam Schiavone and Jeroen Sijsling for discussions.  Voight was supported by a Simons Collaboration grant (550029).

\section{Group theory and geometry} \label{sec:background}

In this section, we begin by developing some preliminary input coming from group theory and geometry.  

\subsection{Transitive permutation representations}
\label{sec:TPRs}

First, a few basic facts and conventions.  In this article, the symmetric group $S_d$ acts on the right on $\{1,\dots,d\}$, written in exponentiated form: e.g., if $\tau = (1\,2\,3)$ and $\mu = (2\,3)$ then $1^{\tau\mu} = (1^\tau)^{\mu} = 2^{\mu} = 3$.

Recall that if $G$ is a group, a \defi{(finite) permutation representation} of $G$ is a group homomorphism $\pi \colon G \to S_d$ for some $d \geq 1$, and we say that $\pi$ is \defi{transitive} if its image is a transitive subgroup of $S_d$.  A transitive permutation triple $\sigma$ defines a transitive permutation representation by $\pi(\delta_s)=\sigma_s$ for $s=a,b,c$, and conversely.  

Let $\pi \colon \Delta \to S_d$ be a transitive permutation representation.  Let
\begin{equation}
\Gamma  \colonequals \{\delta \in \Delta : 1^{\pi(\delta)}=1\}.
\end{equation}
be the preimage of the stabilizer of $1$ under $\pi$.  (The stabilizer of $k \in \{1,\dots,d\}$ is conjugate to $\Gamma$ in $\Delta$.)  Then $[\Delta:\Gamma]=d$.  Conversely, given $\Gamma \leq \Delta$ of index $d$, the action of $\Delta$ on the cosets of $\Gamma$ gives a transitive permutation representation $\pi \colon \Delta \to S_d$, and this correspondence is bijective.  

\subsection{Euclidean triangle groups} \label{sec:TriGroups}

We refer to Magnus \cite[\S II.4]{Magnus} for classical background on Euclidean triangle groups; we briefly summarize some classical facts.  Let $T^*$ be a triangle in the Euclidean plane $\C$ with angles $\pi/a$, $\pi/b$, and $\pi/c$ at the vertices $v_a$, $v_b$, and $v_c$ labeled clockwise, with $a,b,c \in \Z_{\geq 2}$.  Then in fact there are only three possibilities, namely 
\[ (a,b,c)=(3,3,3),(2,3,6),(2,4,4)\] 
corresponding to the solutions to $1/a+1/b+1/c=1$; the corresponding tessellations of the Euclidean plane by triangles are sketched in Figure \ref{fig:tesseuc}, with alternating triangles colored white and black.  

\begin{equation} \label{fig:tesseuc}\addtocounter{equation}{1} \notag
\begin{gathered}
\includegraphics[scale = .25]{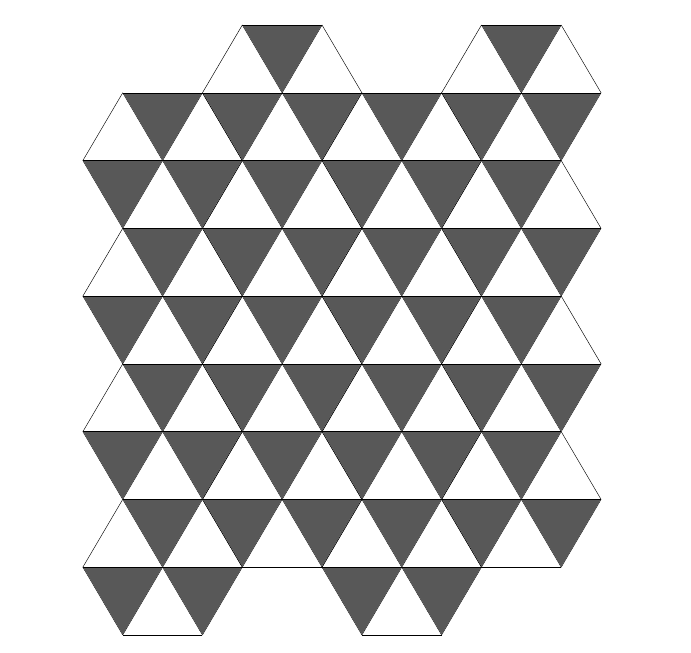}
\includegraphics[scale = .25]{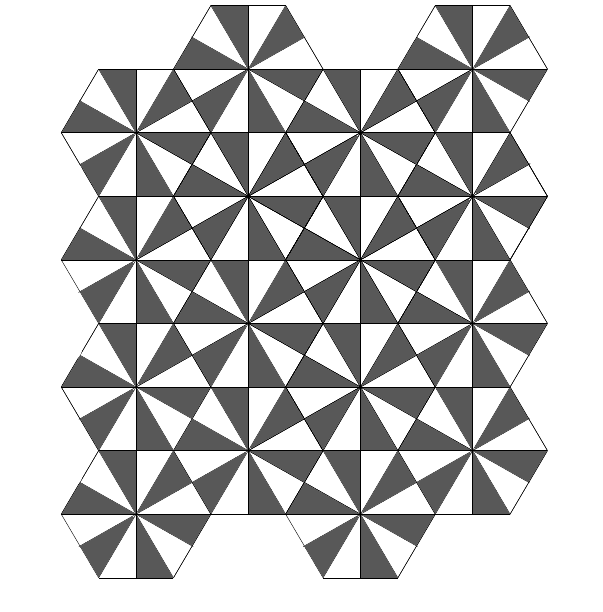}
\includegraphics[scale = .25]{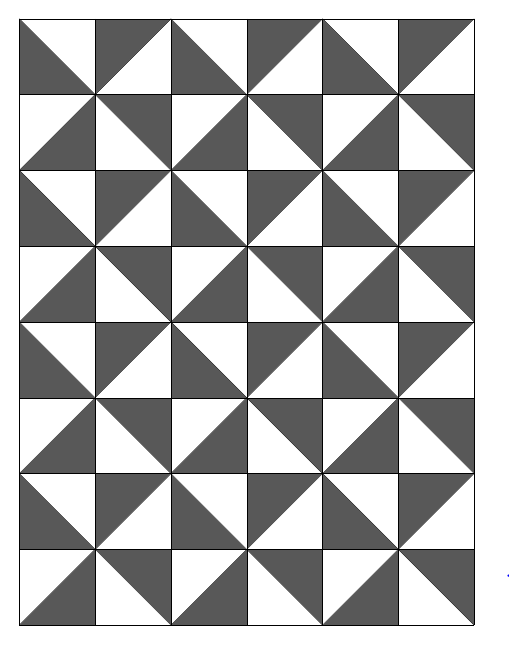}
 \\[4pt]
\text{Figure \ref{fig:tesseuc}: Tessellations for $\Delta(3,3,3)$, $\Delta(2,3,6)$, and $\Delta(2,4,4)$}
\end{gathered}
\end{equation}

The group generated by the reflections in the sides of $T^*$ generates a discrete group of isometries acting properly on $\C$, with fundamental domain $T^*$.  The further subgroup of orientation-preserving isometries $\Delta(a,b,c)$ has index $2$, described as follows.  For $s \in \{a,b,c\}$, let $\delta_s$ be the counterclockwise rotation about $v_s$ by an angle of $2\pi/s$.  

\begin{prop}  \label{prop:TDNormal}
The following statements hold.
\begin{enumalph}
\item There is a presentation
\[ \Delta = \Delta(a,b,c) \simeq \langle \delta_a, \delta_b, \delta_c\, |\, \delta_a^a = \delta_b^b = \delta_c^c = \delta_c\delta_b\delta_a = 1 \rangle. \]
\item There is a unique group homomorphism 
\begin{equation}
\rho \colon \Delta \to \tfrac{1}{c}\Z/\Z \xrightarrow{\sim} \Z/c\Z 
\end{equation}
such that 
\[ \delta_a,\delta_b,\delta_c \mapsto 1/a,1/b,1/c \mapsto c/a,c/b,1. \]
\item We have $\ker \rho = T(\Delta)$ where $T(\Delta) \trianglelefteq \Delta$ is the subgroup of translations, giving a split exact sequence
\begin{equation} \label{eqn:exactseq!} 
1 \to T(\Delta) \to \Delta \xrightarrow{\rho} \langle \Z/c\Z \rangle \to 1;
\end{equation}
in particular,
\[ \Delta=T(\Delta)\langle \delta_c \rangle \simeq \Z^2 \rtimes \Z/c\Z. \]
\end{enumalph}
% then $\delta$ is of the form $z \mapsto \alpha z + \beta$ where $z \mapsto \alpha z$ is a rotation in $\langle \delta_c \rangle$ and $z \mapsto z + \beta$ is a translation in $T(\Delta)$.
\end{prop}

\begin{proof}
See Magnus \cite[Theorem 2.5]{Magnus} for a proof of (a) using the Reidemeister--Schreier method.  

For part (b), we check that the relations in $\Delta$ are satisfied: indeed, we have $\rho(\delta_s^s)=s(c/s) \equiv 0 \pmod{c}$, and $\rho(\delta_c\delta_b\delta_a)=c(1/a+1/b+1/c) \equiv 0 \pmod{c}$.  Alternatively, we define a group homomorphism first by taking the quotient by the commutator subgroup to surject onto $(\Z/a\Z \oplus \Z/b\Z \oplus \Z/c\Z)/\langle (1,1,1) \rangle$, then map to $\Z/c\Z$ via $(x,y,z) \mapsto x(c/a)+y(c/b)+z$.

Since it will be of some importance to us, we prove part (c) two ways.  First, we compute algebraically.  We treat the case $\Delta=\Delta(2,3,6)$, the other two being similar. Without loss of generality, we may suppose that $v_c=0$ and $v_b= 1$. Then $v_a=(\zeta_6+1)/2$, where $\zeta_6=\exp(2\pi i/6)$.  The translations in $\Delta$ are precisely those that translate by the $\Delta$ orbit of $v_c=0$, so $T(\Delta)$ is generated by $z \mapsto z + (\zeta_6+1)=z+2v_a$ and $z \mapsto z + \sqrt{3}i=z+(2\zeta_6-1)$.  We then compute directly that 
\[ \delta_a(z) = -z+(\zeta_6+1) \]
is the composition of the rotation $z \mapsto -z=\zeta_6^3 z$ in $\langle \delta_c \rangle$ followed by the translation 
$z \mapsto z + (\zeta_6 + 1)$ in $T(\Delta)$.  Since $\delta_b=\delta_c^{-1}\delta_a^{-1}$, we conclude that $\Delta=T(\Delta)\langle \delta_c \rangle$.  In particular, every transformation $\delta \in \Delta$ is of the form $\delta(z)=\zeta_6^i z + \beta$ for $i \in \Z/c\Z$ and with $z \mapsto z+\beta$ in $T(\Delta)$; and written this way, $\rho(\delta) = i \pmod{c}$, so indeed $\ker \rho = T(\Delta)$.  
(We may also verify independently that $T(\Delta)$ is normal in $\Delta$: if $\tau(z)=z+\beta \in T(\Delta)$ then
\begin{equation} 
(\delta_c^{-1} \tau \delta_c)(z)=z+\zeta_6^{-1}\beta = z+\delta_c^{-1}(\beta)
\end{equation}
is again translation by a point in the $\Delta$ orbit of $v_c$, so $\delta_c^{-1} \tau \delta_c \in T(\Delta)$.)  Finally, since $T(\Delta) \simeq \Z^2$ is generated freely by two translations, it follows that $\Delta \simeq \Z^2 \rtimes \Z/c\Z$ as claimed.

We may also argue geometrically, as follows.  Intuitively, each transformation $\delta_s$ rotates the plane by the corresponding interior angle $2\pi/s=(c/s)(2\pi/c)$, composition accumulates this rotation in an abelian way, and the resulting transformation is a translation if and only if the total amount of rotation sums to a multiple of $2\pi$.  In other words, every element of $\Delta$ is obtained by first rotation by a power of $\delta_c$ to put $E$ into one of $c$ positions, then translation of $E$: see Figure \ref{fig:deltalemmapng}.

\begin{equation} \label{fig:deltalemmapng}\addtocounter{equation}{1} \notag
\begin{gathered}
\includegraphics [scale =.3]{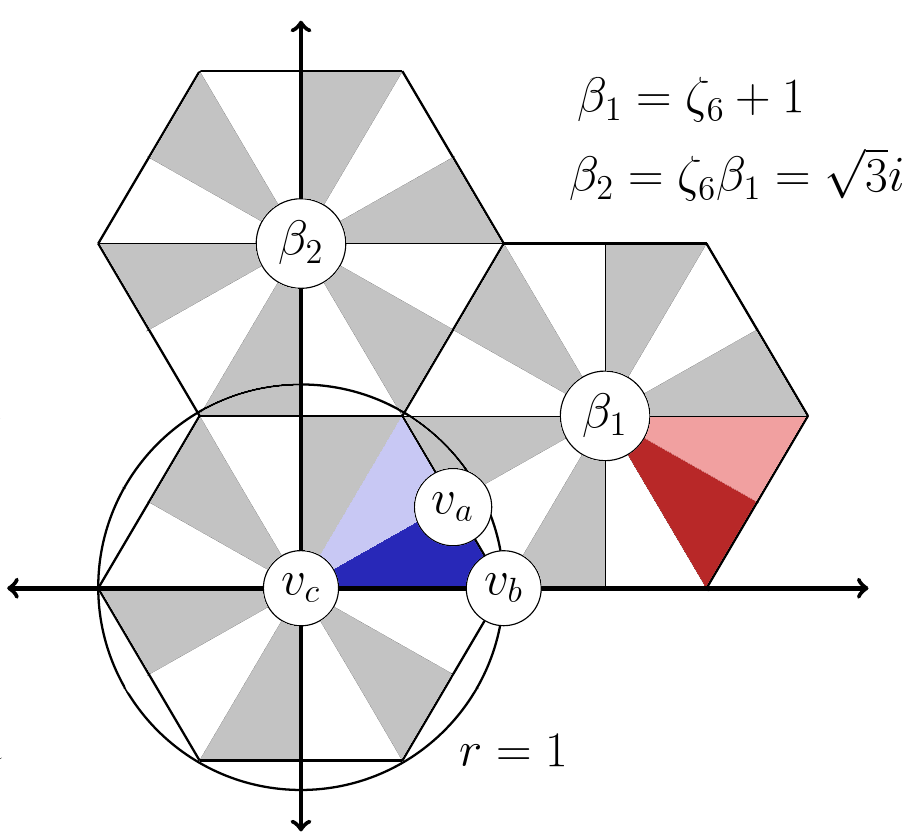}
\\[4pt]
\text{Figure \ref{fig:deltalemmapng}: Geometric proof of Proposition \ref{prop:TDNormal}(b)--(c)}
\end{gathered}
\end{equation}

More precisely, around $v_c$ there is a central hexagon or square $E$ consisting of $c$ pairs of white and black triangles.  Let $\delta \in \Delta$.  Then $\delta(E)$ is another hexagon or square in the tessellation, with center $\delta(v_c)$.  It is geometrically evident (and can be verified in a straightforward manner) that there is a unique translation $\tau_\delta \in T(\Delta)$ mapping $\delta(v_c)$ to $v_c$, so the composition fixes $v_c$ and maps $E$ to itself.  But again visibly, the stabilizer of $E$ in $\Delta$ is precisely $\langle \delta_c \rangle$.  This association thereby defines a surjective group homomorphism $\Delta \to \langle \delta_c \rangle$ with kernel $T(\Delta)$, as claimed.  Figure \ref{fig:deltalemmapng} gives the transformation taking $T^*$ (blue) to $T'$ (red) by first rotating about $v_c$ by $5\pi/3$ (applying $\delta_c^5$) then translating by $z \mapsto z + \beta_1$, an element of $T(\Delta)$.  
\end{proof}

\begin{corollary}  \label{cor:Tgens}
The group $T(\Delta)$ is generated by 
\begin{equation}
(\omega_1,\omega_2) \colonequals
\begin{cases}
(\delta_a \delta_c^2, \delta_b \delta_c^2), & \text{ if $(a,b,c)=(3,3,3)$;} \\
(\delta_a\delta_c^3,\delta_b\delta_c^4), & \text{ if $(a,b,c)=(2,3,6)$;} \\
(\delta_a \delta_c^2, \delta_b \delta_c^3), & \text{ if $(a,b,c)=(2,4,4)$.}
\end{cases}
\end{equation}
\end{corollary}

\begin{proof}
In each case, $\omega_1$ and $\omega_2$ are in the kernel of the homomorphism $\rho$ described in proposition $\ref{prop:TDNormal}$, and thus $\omega_1, \omega_2 \in T(\Delta)$. From Figure \ref{fig:tesseuc}, it is straightforward to verify that the $\langle \omega_1, \omega_2 \rangle$ orbit of $v_c$ is the same as the $T(\Delta)$ orbit of $v_c$, so $T(\Delta) = \langle \omega_1, \omega_2 \rangle$.
\end{proof}

Visibly from Figure \ref{fig:tesseuc} we have $\Delta(3,3,3) \trianglelefteq \Delta(2,3,6)$ with index $2$ (halving a fundamental triangle), and $T(\Delta(3,3,3))=T(\Delta(2,3,6))$.  Attached to each translation subgroup is the orbit of $0$ 
\begin{equation}
\Lambda_{\Delta} \colonequals T(\Delta)\cdot 0 
\end{equation}
which defines a lattice $\Lambda_{\Delta} \simeq \Z^2$.  We write
\begin{equation} \label{eqn:squarehex}
\begin{aligned}
\Lambda_{\square} &\colonequals \Lambda_{\Delta(2,4,4)} = \Z[i] \\
\Lambda_{\hexagon} &\colonequals \Lambda_{\Delta(3,3,3)} = \Lambda_{\Delta(2,3,6)} = \Z[\zeta_6]
\end{aligned}
\end{equation}

\begin{remark} \label{rmk:homothetylattice}
More precisely, we work with these lattices up to homothety, rescaling by an element of $\C^\times$; to obtain elliptic curves defined over $\Q$ (see section \ref{sec:F}), we must rescale by a real number (which can be given explicitly as a real period).
\end{remark}

\subsection{Fundamental domains} \label{sec:quots}

In this section, we describe fundamental domains for the groups under consideration.  
% If we take a subgroup $H$ of $\Delta$ and identify points in $\C$ which can be carried to each other via an element of $H$ (i.e. $z_1 \approx z_2$ if and only if $z_1 = h(z_2)$ for some $h \in H$), taking a complete set of representatives of the equivalence classes gives a \defi{fundamental domain} contained in $\C$. In particular, we are concerned with the cases of $\C/\Delta, \C/\Gamma,$ $\C/T(\Delta),$ and $\C/T(\Gamma)$.
A fundamental domain for the action of $\Delta$ is obtained from any pair of one shaded triangle and one unshaded triangle which we may take to share an edge. This gives a region where all the interior points are distinct under the identification $\Delta \circlearrowright \C$. Furthermore, we can divide the four sides of the quadrilateral into two pairs of consecutive sides 
% $(a_1,a_2)$ and $(b_1,b_2)$ such that $a_1$ and $a_2$ are 
identified under the quotient by $\Delta$ as in Figure \ref{fig:DeltaFundoms},  
% and and $b_1$ and $b_2$ are identified (see diagram below). 
so that $X(\Delta)$ has genus 0.  % , $\PoneC$ with three cone points at the vertices.
\begin{equation} \label{fig:DeltaFundoms}\addtocounter{equation}{1} \notag
\begin{gathered}
\includegraphics[scale = .2]{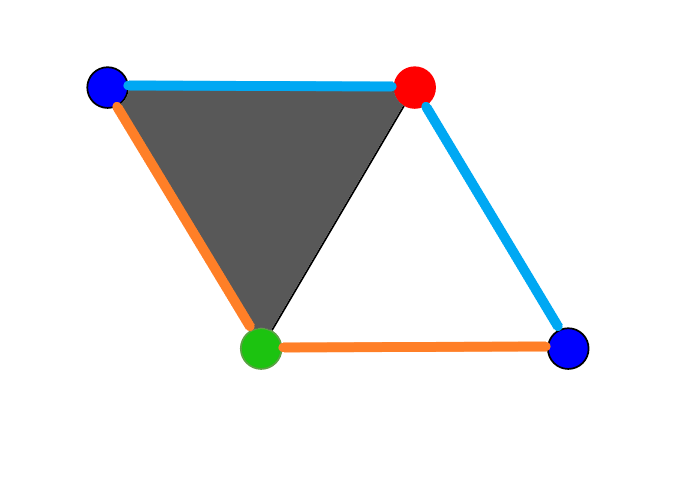}
\includegraphics[scale = .2]{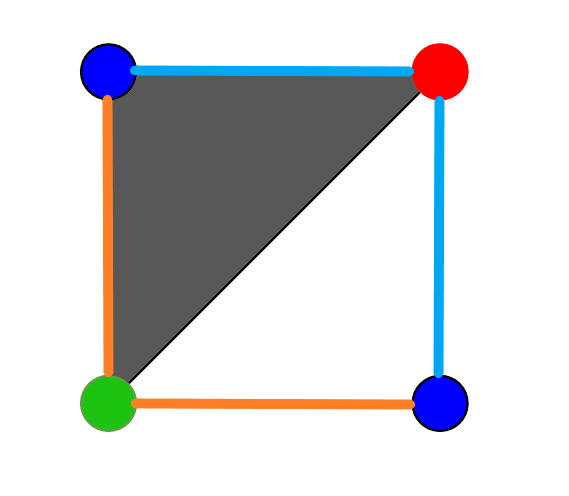}
\includegraphics[scale = .2]{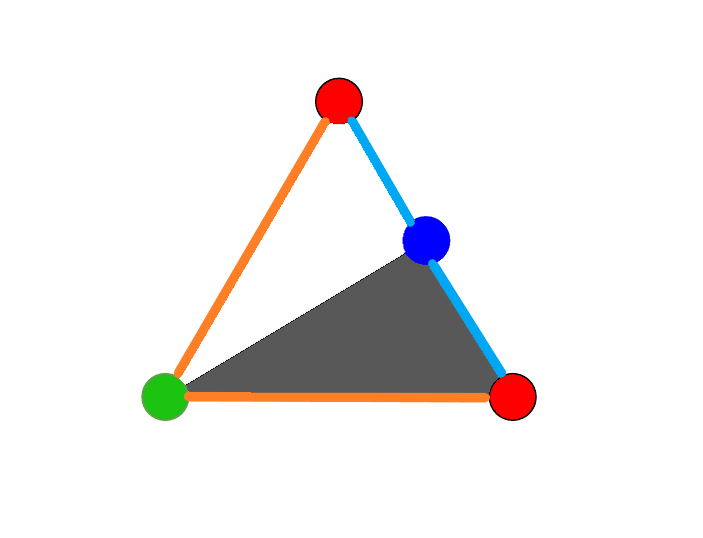}
\\[4pt]
\text{Figure \ref{fig:DeltaFundoms}: Fundamental domains for $\Delta$, like colored edges and vertices are identified}
\end{gathered}
\end{equation}

Since $T(\Delta)$ is generated by two noncollinear translations, we can take as its fundamental domain the parallelogram determined from two sides sharing a vertex at the origin. Opposite edges are identified while consecutive edges are distinct, so the fundamental region is equivalent to a torus (genus 1). Similar statements hold for $T(\Gamma)$.  
\begin{equation} \label{fig:FundomsTD}\addtocounter{equation}{1} \notag
\begin{gathered}
\includegraphics[scale = .32]{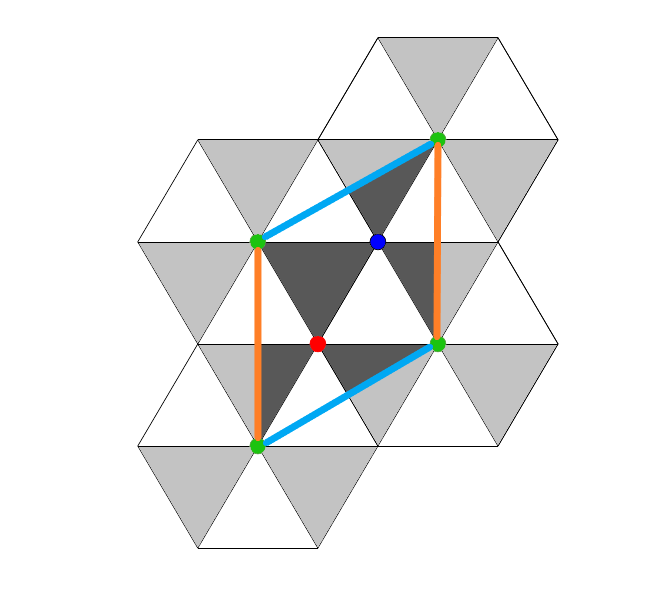}
\includegraphics[scale = .30]{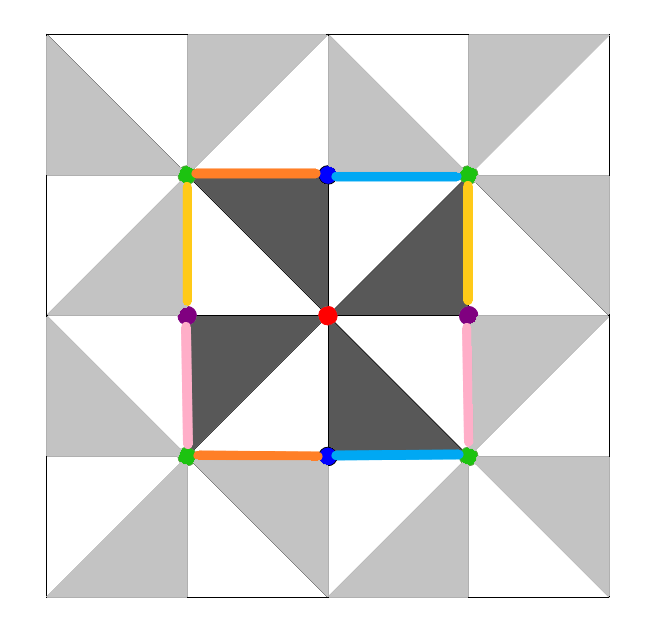}
\includegraphics[scale = .24]{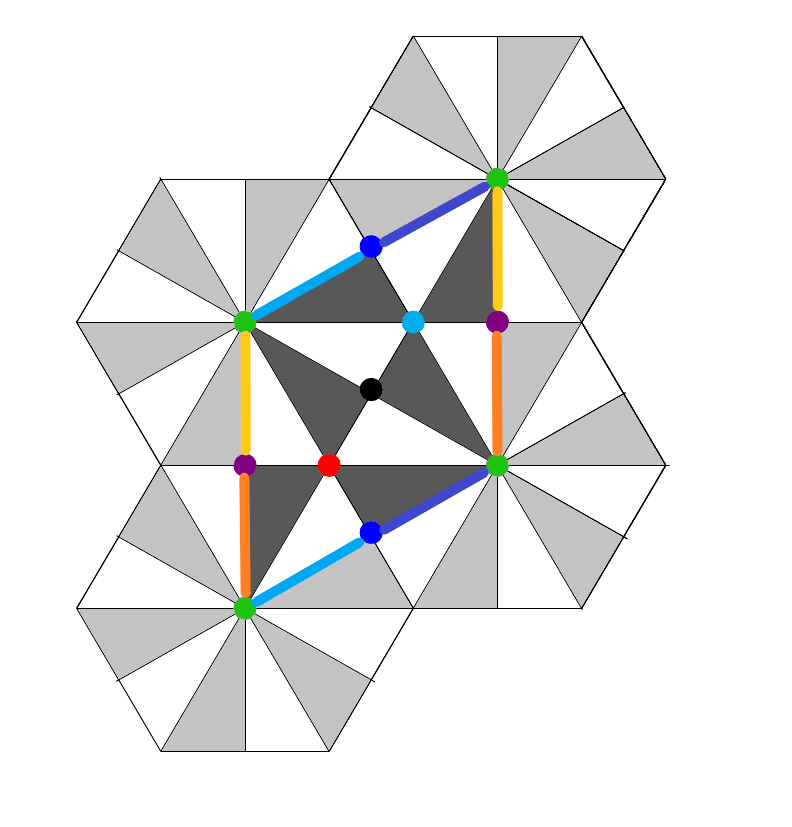}
\\[4pt]
\text{Figure \ref{fig:FundomsTD}: Fundamental domains for $T(\Delta)$, like colors identified}
\end{gathered}
\end{equation}

Finally, a fundamental domain for $\Gamma$ is constructed in the usual manner: we choose coset representatives $\Delta = \bigsqcup_{i=1}^d \gamma_i \Gamma$, and then for $D(\Delta)$ the fundamental domain for $\Delta$ we have the fundamental domain $D(\Gamma)=\bigcup_{i=1}^d \gamma_i D(\Delta)$. 

We now consider the genus of the surface $X(\Gamma) \colonequals \Gamma \backslash \C$.  Given a permutation $\tau \in S_d$, let $k(\tau)$ be the number of disjoint cycles in $\tau$ and define its \defi{excess} as $e(\tau) \colonequals d-k(\tau)$.  Then by the Riemann--Hurwitz formula, the genus of $X(\Gamma)$ is equal to \cite[(1.5)]{SijslingVoight}
\begin{equation} \label{eqn:RH}
g(X(\Gamma)) = 1-d + \frac{e(\sigma_0)+e(\sigma_1)+e(\sigma_\infty)}{2}. 
\end{equation}

\begin{lemma} \label{lemma:Genus}
% For a finite index subgroup $\Gamma$ of $\Delta$ obtained as described, 
We have $g(X(\Gamma)) \leq 1$, with equality if and only if for all $s \in \{a,b,c\}$, every cycle in $\sigma_s$ has length $s$.  
\end{lemma}

\begin{proof}
For $s \in \{a,b,c\}$ since the cycle decomposition of $\sigma_s$ can contain no cycle of length greater than $s$, we have $k(\sigma_s) \ge d/s$, so
\[ e(\sigma_a) + e(\sigma_b) + e(\sigma_c) \le 3d - \left(\frac{d}{a} + \frac{d}{b} + \frac{d}{c}\right) = 3d -d = 2d \]
with equality if and only if all cycles in $\sigma_s$ are length $s$.  Substituting this into \eqref{eqn:RH}, the result follows.
\end{proof}

\begin{remark}
We will see later, in Corollary \ref{corr:genus1}, that $g(X(\Gamma)) = 1$ if and only if $\Gamma = T(\Gamma)$.
\end{remark}

\subsection{Translation subgroups} \label{sec:TransSubgps}

Let 
\begin{equation}
T(\Gamma) \colonequals \Gamma \cap T(\Delta) = \ker \rho|_{\Gamma} 
\end{equation}
be the subgroup of translations in $\Gamma$; then $T(\Gamma) \trianglelefteq \Gamma$, as  $T(\Delta) \trianglelefteq \Delta$ by Proposition \ref{prop:TDNormal}.  Writing $E(\Gamma) \colonequals T(\Gamma) \backslash \C$ and similarly for $\Delta$, the containments of these four groups give quotient maps which fit into the diagram \eqref{eqn:EXDeltaGamma}.

We again have a lattice 
\begin{equation} \label{eqn:lamgam}
\Lambda_{\Gamma} \colonequals T(\Gamma)\cdot 0 
\end{equation}
with $\Lambda_{\Gamma} \leq \Lambda_{\Delta}$ a subgroup of finite index.  When no confusion can arise, we will identify translation maps by the corresponding lattice element.  We define
\begin{equation}
N \colonequals [T(\Delta):T(\Gamma)].
\end{equation}

In the following algorithm, we compute a convenient basis for $T(\Gamma)$.

\begin{algorithm} \label{alg: basisForTG}
This algorithm takes as input $\sigma$ and outputs a basis $\eta_1,\eta_2$ for $T(\Gamma)$ and $N =[T(\Delta):T(\Gamma)]$.  

\begin{enumalg}
\item Let $\pi$ be the transitive permutation representation attached to $\sigma$, and for $i=1,2$, let $\omega_i$ be as in Corollary \ref{cor:Tgens} (a basis for $T(\Delta)$).

\item Let $\tau_1$ be the cycle containing $1$ in $\pi(\omega_1)$ and let $\tau_2$ be the cycle containing $1$ in $\pi(\omega_2^{-1})$.  For $i=1,2$, let $\ell_i$ be the length of $\tau_i$.

\item Compute  
\[ V \colonequals \bigl\{(b_1, b_2) : 0 \leq b_i \leq \ell_i \text{ for $i=1,2$ and } 1^{\tau_1^{b_1}} = 1^{\tau_2^{b_2}}\bigr\}. \]

\item Let $A$ be the matrix whose rows are the elements of $V$.  Reduce $A$ to Hermite normal form (HNF) and take its first two row vectors $(n_1, n_2)$ and $(0, m_2)$. 

\item Return $\eta_1 = \omega_1^{n_1} \omega_2^{n_2}$ and $\eta_2 = \omega_2^{m_2}$ and $N=n_1m_2$.
\end{enumalg}
\end{algorithm}

\begin{proof}[Proof of correctness]
Since $\omega_1$ and $\omega_2$ commute, any $\eta \in T(\Delta) $ is of the form $\eta = \omega_1^{a_1}\omega_2^{a_2}$ for some $(a_1, a_2) \in \Z^2$. By definition, such $\eta \in T(\Gamma)$ if and only if $1^{(\pi_1(\omega_1)^{a_1}\pi_2(\omega_2)^{a_2})}=1$, or equivalently when $1^{\tau_1^{a_1}} = 1^{{\tau_2^{a_2}}}$.  Since $\tau_i$ has order $\ell_i$, we only need to consider $0 \leq a_i \leq \ell_i$ for $i=1,2$.  The $\Z$-span of $V$ therefore gives all pairs $(a_1,a_2)$ such that $\eta = \omega_1^{a_1}\omega_2^{a_2}$ is in $T(\Gamma)$.  Since only row operations are performed in computing the Hermite normal form, the $\Z$-span does not change, hence $\eta_1,\eta_2$ computed in step 5 generate $T(\Gamma)$.  Finally, we have 
\[ N=[T(\Delta):T(\Gamma)]=\det \begin{pmatrix} n_1 & n_2 \\ 0 & m_2 \end{pmatrix} = n_1m_2. \qedhere \]
\end{proof}

\subsection{Rotation index}
\label{sec:RotInd}

In this section, we study rotations in $\Gamma$.  Restricting the exact sequence \eqref{eqn:exactseq!} we obtain
\[ 1 \to T(\Gamma) \to \Gamma \to R(\Gamma) \to 1 \]
where $R(\Gamma) \colonequals \rho(\Gamma) \leq \Z/c\Z$.  Evidently, $R(\Gamma)$ is a cyclic group with order dividing $c$.

\begin{defn} \label{defn:RotInd}
The \defi{rotation index} of $\Gamma$ is $r(\Gamma) \colonequals [\Gamma:T(\Gamma)]=\#R(\Gamma)$.
\end{defn}

\begin{lemma} \label{lem:RotInd} 
We have
\[ r(\Gamma) = \frac{cN}{d} \]
where $N=[T(\Delta):T(\Gamma)]$.
\end{lemma}

\begin{proof}
From
\[ [\Delta : T(\Gamma)] = [\Delta:\Gamma][\Gamma:T(\Gamma)]=[\Delta:T(\Delta)][T(\Delta):T(\Gamma)] \]
we conclude $dr(\Gamma)=cN$.
\end{proof}

In Proposition \ref{prop:TDNormal}(c) we split the exact sequence using $\delta_c$.  Indeed, the analogous sequence for $\Gamma$ above is again split, but not necessarily by a power of $\delta_c$: instead, $R(\Gamma)$ is generated by a rotation about some vertex (an element in the $\Delta$ orbit of $v_a$, $v_b$, or $v_c$), as follows.

\begin{lemma} \label{lemma:RGgenByVOMR}
There exists a vertex $v_O$ whose stabilizer $\gamma_O \in \Gamma$ has $\rho(\gamma_O)$ a generator of $R(\Gamma)$, giving a split exact sequence
\[ 1 \to T(\Gamma) \to \Gamma \to \langle \gamma_O \rangle \to 1 \]
so in particular $\Gamma = T(\Gamma) \langle \gamma_O \rangle \simeq \Z^2 \rtimes \Z/r(\Gamma)\Z$.  
\end{lemma}

\begin{proof}
Every element of $\Delta$ is either a translation (and fixes no point) or fixes a unique point ($z \mapsto uz+v$ fixes $z=v/(1-u)$ if $u \neq 1$), necessarily a vertex as every nonidentity element of finite order in $\Delta$ is conjugate to one of the generators $\delta_a,\delta_b,\delta_c$.  So let $\gamma_O \in \Gamma$ be any element which maps to a generator of $R(\Gamma)$ under $\rho$, well-defined up to a translation in $T(\Gamma)$.  If  $\gamma_O$ is a translation, which is to say $\gamma_O \in T(\Gamma)$, then $R(\Gamma)$ is trivial: hence $\Gamma=T(\Gamma)$, and we may take $v_O$ to be any vertex (each having trivial stabilizer under $\Gamma$).  

Otherwise, $\gamma_O$ fixes a vertex $v_O$ with the claimed properties; the splitting follows immediately, just as we saw in the geometric proof of Proposition \ref{prop:TDNormal}(c).
\end{proof}

\begin{defn} \label{defn:maxstab}
A vertex $v_O$ whose stabilizer generates $R(\Gamma)$ is called a \defi{vertex of maximum rotation}.
\end{defn}

With Lemma \ref{lemma:RGgenByVOMR}, we can be more precise about the possible vertices of maximal rotation.  

\begin{corollary}  \label{cor:maxtgamm}
The vertices of maximal rotation, up to translation by $T(\Gamma)$, are in bijection with the union of the sets of cycles $\tau$ in $\sigma_s$ with length $s/r(\Gamma)$ for $s \in \{a,b,c\}$.
\end{corollary}

\begin{proof}
Under the quotient map $\Gamma \backslash \C \to \Delta \backslash \C$, for $s \in \{a,b,c\}$, the preimages of the vertex $v_s$ are in bijection with the cycles in $\sigma_s$ and the stabilizer of a vertex with cycle $\tau$ has order $s/\ell(\tau)$ where $\ell(\tau)$ is the length of $\tau$.  Such a vertex has maximal rotation if and only if $s/\ell(\tau)=r(\Gamma)$.
\end{proof}

Because a permutation triple which is simultaneously conjugate to $\sigma$ gives an isomorphic Belyi map (with differently labelled sheets), we may suppose without loss of generality that one of $v_a,v_b,v_c$ is a vertex of maximal rotation: after simultaneous conjugation, we just insist that $1$ belongs to a cycle as in Corollary \ref{cor:maxtgamm}.  This ``preprocessing'' step is given as follows.

\begin{algorithm}\label{alg:preproc}
This algorithm takes as input a Euclidean permutation triple $\sigma$ and gives as output the rotation index $r(\Gamma)$ and a simultaneously conjugate triple $\sigma'$ and $s \in \{a,b,c\}$ such that one of $v_a,v_b,v_c$ is a vertex of maximal rotation
\begin{enumalg}
\item Compute $N$ using Algorithm \ref{alg: basisForTG} and $r(\Gamma)=cN/d$.  
\item By trying all possibilities, find a cycle $\tau$ in $\sigma_s$ with $s\in \{a,b,c\}$ with length $\ell(\tau)=s/r(\Gamma)$.  
\item For any $i \in \tau$, return $r(\Gamma)$ and the simultaneous conjugation of $\sigma$ by $(1\,i)$.
\end{enumalg}
\end{algorithm}

\begin{proof}
In Step 1, the rotation index is computed correctly by Lemma \ref{lem:RotInd}.  Step 2 will succeed by \ref{cor:maxtgamm}.  By choice of $\Gamma$ as the stabilizer of $1$, we conclude that $v_s$ is a vertex of maximal rotation. 
\end{proof}

From here forward, we may suppose without loss of generality that this ``preprocessing'' step has been applied.

We now see the exact circumstances when $g(X(\Gamma)) = 1$.

\begin{corr} \label{corr:genus1} 
We have $g(X(\Gamma)) = 1$ if and only if $r(\Gamma)=1$ if and only if $\Gamma=T(\Gamma)$.  
\end{corr}

\begin{proof}
By Corollary \ref{cor:maxtgamm}, we have $r(\Gamma)=1$ if and only if for all $s \in \{a,b,c\}$, every cycle in $\sigma_s$ has length $s$; the result then follows from Lemma \ref{lemma:Genus}.
\end{proof}

\section{Equations} \label{sec:mainresult}

From the subgroup $\Gamma \leq \Delta$ of index $d$, in the previous section we defined the translation subgroups $T(\Gamma) \leq T(\Delta)$ whose quotients fit into the commutative diagram \eqref{eqn:EXDeltaGamma}.  We now calculate equations for these curves and the maps between them.  As a basic reference, we refer to Silverman \cite{Silverman,Silverman2}.

\subsection{Fixed maps} \label{sec:F}

We begin with the bottom map $\alpha \colon E(\Delta) \to X(\Delta) \simeq \PP^1$, which depends only on $\Delta$ (with the choice of the origin at $v_c$).  From Proposition \ref{prop:TDNormal}(c), the map $\alpha$ is the quotient by a cyclic group of rotations of order $c$ at a vertex $v_c$, which we may take as the origin of the elliptic curve $E(\Delta)$.  Accordingly, these rotations act by automorphisms of the elliptic curve $E(\Delta)$, and so their equations are well-known \cite[\S II.2]{Silverman2} (see also Lemma \ref{lemma:zeroCoeffs} below).  Define the elliptic curves
\begin{equation}  \label{Ehexsquare}
E_{\hexagon} \colon y^2 = x^3 + 1 \hspace{5ex} E_{\square} \colon y^2 = x^3 - x 
\end{equation}
over $\Q$, the automorphisms
\begin{equation} \label{eqn:delta346}
\begin{aligned}
\delta_3 \colon E_{\hexagon} &\to E_{\hexagon} \hspace{5ex} & \delta_4 \colon E_{\square} &\to E_{\square} \hspace{5ex} & \delta_6 \colon E_{\hexagon} &\to E_{\hexagon} \\
(x,y) &\mapsto (\zeta_3 x,y) & (x,y) &\mapsto (-x, iy) & (x,y) &\mapsto (\zeta_3^{-1}x, -y).
\end{aligned}
\end{equation}
and the quotient maps
\begin{equation} \label{eqn:alpha346}
\begin{aligned}
\alpha_3 \colon  E_{\hexagon} &\to \PP^1 \hspace{5ex} & \alpha_4 \colon  E_{\square} &\to \PP^1 \hspace{5ex}  & \alpha_6 \colon  E_{\hexagon} &\to \PP^1 \\
(x,y) &\mapsto \frac{y+1}{2} & (x,y) &\mapsto x^2 & (x,y) &\mapsto y^2.
\end{aligned}
\end{equation}
We recall the lattices defined in \eqref{eqn:squarehex}.  After homothety, the Weierstrass map $z \mapsto (\wp(z),\wp'(z)/2)$ gives an analytic isomorphism from the complex elliptic curve $\C/\Lambda_{\square}$ to $E_{\square}(\C)$.  Moreover, the rotation $\delta_4$ acts by $(x,y) \mapsto (-x,iy)$ (gently abusing notation), and the quotient map $E(\Delta) \to X(\Delta)$ is given by $\alpha_4$ in these coordinates.  Similar statements hold for the two $\hexagon$ cases.  

\begin{lemma}
The maps $\alpha_c$ for $c=3,4,6$ are Euclidean Belyi maps of degree $c$.
\end{lemma}

\begin{proof}
For $\alpha_4$, the set of preimages under $(x,y) \mapsto x^2=t$ has cardinality four unless $t=0,\infty$ or $y=0$, in which case $t=x^2=0,1$, giving ramification type $(2,4,4)$.  Similarly for $\alpha_6$, we have six preimages under $(x, y) \mapsto y^2=t$ unless $t=0,\infty$ or $y^2-1=x^3=0$, in which case $t=y^2=1$, giving ramification $(2,3,6)$.

For $\alpha_3$, the map $(x,y) \mapsto y$ is ramified above $\{\pm 1,\infty\}$ with ramification $(3,3,3)$, so to get ramification at $\{0,1,\infty\}$ we simply postcompose with the M\"obius transformation $y \mapsto (y+1)/2$.
\end{proof}

\subsection{Isogeny} % from  \texorpdfstring{$E(\Gamma)$}{E(Gamma)} to \texorpdfstring{$E(\Delta)$}{E(Delta)}} 
\label{sec:psi}

We now turn to the left map $\psi$ in \eqref{eqn:EXDeltaGamma}, an isogeny $E(\Gamma) \to E(\Delta)$.  

We first show how to work explicitly with torsion on $E(\Delta)$ using exact arithmetic.  To handle the three cases uniformly, let $j=i$ or $j=\zeta_6$, so that $\Lambda=\Lambda_{\Delta}=\Z[j]$, let $K \colonequals \Q(j) \subseteq \C$, and let $E=E_{\square}$ or $E=E_{\hexagon}$.

\begin{lemma} \label{lem:abj}
For all $a+bj \in \Z[j]$, there exists an effectively computable rational function $m_{a+bj}(x) \in K(x)$ such that $x([a+bj]P) = m_{a+bj}(x(P))$ for all $P \in E(\Qbar)$.
\end{lemma}

\begin{proof}
When $b=0$, the lemma is established as part of the theory of division polynomials: see Silverman \cite[Exercise 3.7(d)]{Silverman}.  When $b \ne 0$, we may similarly calculate using the explicit description of the action of $j$ given in \eqref{eqn:delta346}: we still know that $x([a+bj]P)$ is a rational function in $x(P)$, since $x([a+bj](-P))=x(-[a+bj]P)=x([a+bj]P)$.  
\end{proof}

For an integer $N \geq 1$, the torsion group 
\[ E[N] \simeq \tfrac{1}{N}\Lambda/\Lambda \simeq \Z[j]/N\Z[j] \]
is a cyclic $\Z[j]$-module; we use the symbol $P \in E[N]$ to denote a generator of $E[N]$ as a $\Z[j]$-module.

\begin{algorithm}\label{alg:genP}
This algorithm takes as input $N \in \Z_{\geq 1}$ and returns as output a number field $L$ and the set
\[ \{(a+bj, x([a+bj]P)) : a,b \in \Z/N\Z\} \subseteq \Z[j]/N\Z[j] \times L \]
for a generator $P$.
\begin{enumalg}
\item Compute the $N$-division polynomial $f_N(x) \in \Q[x]$ for $E$.
\item For each proper divisor $D \mid N$, compute the $D$-division polynomial $f_D(x) \in \Q[x]$ for $E(\Delta)$ and divide $f_N(x)$ by $\gcd(f_D(x),f_N(x))$ recursively.  
\item Let $g_N(x)$ be an irreducible factor of $f_N(x)$ over $K[x]$ of highest degree and let $L \colonequals K(\theta)$ with $\theta$ a root of $g_N(x)$.  
\item Return the values 
\[ \{(a+bj, m_{a+bj}(\theta)) : a,b \in \Z/N\Z\}. \]
\end{enumalg}
\end{algorithm}

\begin{remark}
As an alternative to Step 4 (in place of computing the rational functions), at the cost of enlarging $L$ to include the $y$-coordinate (if $E\colon y^2 = f(x)$, we just need $\sqrt{f(\theta)}$), we can just compute directly using the group law on $E$.
\end{remark}

\begin{proof}[Proof of correctness]
In Step 1, we form the polynomial whose roots are the $x$-coordinates of the $N$-torsion points, by definition of the division polynomial.  In Step 2, we remove all roots whose order is a proper divisor of $N$; so any remaining root will be the $x$-coordinate of a point with exact order $N$. Some such point $P$ generates $E[N]$ as a $\Z[j]$ module. By lemma \ref{lem:abj}, taking $\theta$ as a root of an irreducible factor of highest degree in step 3 guarantees that $x(E[N]) \subseteq K(\theta)$. The output of Step 4 is correct by Lemma \ref{lem:abj}.
\end{proof}

Next, we recall section \ref{sec:TransSubgps}, where we defined $N \colonequals [T(\Delta):T(\Gamma)]$ and computed in Algorithm \ref{alg: basisForTG} a basis for $\Lambda_{\Gamma}$.  Since $N\Lambda_{\Delta} \subseteq \Lambda_{\Gamma}$, we have an isogeny
\begin{equation}
\begin{aligned}
\widehat{\psi} \colon E(\Delta) &\to E(\Gamma) \\
z &\mapsto Nz
\end{aligned}
\end{equation}
dual to our desired isogeny $\psi$.  From this setup, we compute an equation for $\psi$ using V\'elu's formulas, as in the following algorithm.

\begin{algorithm}\label{alg: isogeny}
This algorithm takes as input a basis 
\begin{equation}
\begin{aligned}
\eta_1 &= n_1 \omega_1 + n_2\omega_2 \\
\eta_2 &= m_2\omega_2 
\end{aligned}
\end{equation} 
for $\Lambda_{\Gamma} \leq \Lambda_\Delta = \Z\omega_1+\Z\omega_2$ and gives as output a model for the isogeny $\psi \colon E(\Gamma) \to E(\Delta)$.
\begin{enumalg}
\item Let $p_1 \colonequals (0, \lceil n_1/2 \rceil)$. If $m_2$ is odd, let $p_2 := (\lfloor m_2/2 \rfloor, n_1)$. If $m_2$ is even, let $p_2 \colonequals (m_2/2, \lfloor n_1/2 \rfloor$).
\item Let 
\[ C \colonequals \{ (t_1,t_2) \in \Z/m_2\Z \times \Z/n_1\Z : p_1 \le (t_1,t_2) \le p_2 \} \]
where $\le$ here indicates the dictionary order.
\item Let
    	\[K \colonequals  \left\{ \tfrac{1}{N}(t_1n_1 , t_1n_2 + t_2m_2) : (t_1, t_2) \in C \right\}. \]

\item Compute
    \[ X \colonequals  \{\wp_{\Lambda}(a_i\omega_1 + b_i\omega_2) : (a_i, b_i) \in K \} \subseteq \C \]
to enough precision to distinguish their values.  

\item Call Algorithm \ref{alg:genP} with output $W_L$.  Embed $L \hookrightarrow \C$, and let $X_L \subseteq L$ be the set of $x$-coordinates in $W_L$ whose embedding into $\C$ matches a value in $X$.  

\item Let 
    \[ p(x) \colonequals  \prod_{k \in X_L} (x-k) \in L[x]. \]
    Let $K'$ be the subfield of $L$ generated (over $\Q$) by the coefficients of $p(x)$.  
\item Using V\'elu's formulas \cite{Velu}, compute the isogeny $\widehat{\psi} \colon E \to E'$ with kernel $p(x) \in K'[x]$ and $E'$ defined over $K'$.
\item Return the dual isogeny $\psi \colon E' \to E$.  
\end{enumalg}
\end{algorithm}

\begin{proof}[Proof of correctness]
Algorithm \ref{alg: basisForTG} gives  $\eta_1 = n_1\omega_1 + n_2\omega_2$ and $\eta_2 = m_2\omega_2$, so $\Lambda_{\Gamma} \subseteq \Lambda_{\Delta}$. Note also that $N\omega_1 = m_2\eta_1-n_2\eta_2$ and $N\omega_2 = n_1\eta_2$, so $N\Lambda_{\Delta} \subseteq \Lambda_{\Gamma}$.

Let $f_N(x)$ be the $N$-division polynomial.  We determine the $x$-coordinates of the points in $\ker \widehat{\psi}$ from among the roots of $f_{N}$. Since $z \in (1/N) \Lambda_{\Gamma}$ if and only if $Nz \in \Lambda_{\Gamma}$, it follows that $\ker(\widehat{\psi}) = (1/N)\Lambda_{\Gamma}/\Lambda_{\Delta} \simeq \Lambda_{\Gamma}/N\Lambda_{\Delta}$ and
\[ \#\ker(\widehat{\psi}) = \#(\Lambda_{\Gamma}/N\Lambda_{\Delta}) = \det\begin{pmatrix}
m_2 & -n_2\\
0 & n_1 \end{pmatrix} = n_1m_2 = N. \]

To list representatives for $\Lambda_{\Gamma}/N\Lambda_{\Delta}$, we proceed as follows: if we identify ordered pairs $(a,b)$ with coordinates relative to the basis $\{\eta_1,\eta_2\}$ for $\Lambda_{\Gamma}$ (i.e., $(a,b)$ indicates the point $a\eta_1 + b\eta_2$), then $(a_1, b_1)$ and $(a_2, b_2)$ are equivalent modulo $N\Lambda_{\Delta}$ if and only if $a_1 - a_2 = im_2$ and $b_1 - b_2 = jn_1-in_2$ for some $i,j \in \Z$. % Thus, if $m_2 \nmid (a_1 - a_2)$ then $(a_1, b_1) \nsim (a_2, b_2)$, and for a fixed $a$, if $n_1 \nmid(b_1 - b_2)$ then $(a, b_1) \nsim (a, b_2)$. 
So the set
\begin{equation}
\{t_1\eta_1 + t_2\eta_2 : 0 \leq t_1 < m_2, 0 \leq t_2 < n_1 \}
\end{equation}
with $N$ elements gives a complete set of coset representatives for $\Lambda_{\Gamma}/N\Lambda_{\Delta}$. It follows then that the set
\begin{equation}
\begin{aligned}
A &\colonequals \{\tfrac{1}{N}(t_1\eta_1 + t_2\eta_2) : 0 \leq t_1 < m_2, 0 \leq t_2 < n_1 \} \\
&\qquad=\{\tfrac{1}{N} t_1(n_1\omega_1 + n_2\omega_2) + \tfrac{1}{N}t_2(m_2\omega_2) : 0 \leq t_1 < m_2, 0 \leq t_2 < n_1 \} \\
&\qquad=\{ \tfrac{1}{N}xn_1 \omega_1 + \tfrac{1}{N}(t_1n_2 + t_2m_2)\omega_2 :0 \leq t_1 < m_2, 0 \leq t_2 < n_1 \}
\end{aligned}
\end{equation}
gives a complete set of coset representatives for $(1/N)\Lambda_{\Gamma}/\Lambda_{\Delta}$. 

We use the Weierstrass $\wp$-function to map the points in the set $z \in A$ to points $P=(\wp(z),\wp'(z)/2) \in E(\C)$ on the algebraic model $E$.  On this model, since $x(-Q)=x(Q)$ for all $Q$ we only need one representative in $A$ up to inverses. Points in $A$ corresponding to the pairs $(t_1, t_2)$ and $(t_1', t_2')$ give inverses on $E(\C)$ if and only if $m_2 \mid (t_1 + t_1')$ and $n_1 \mid (t_2 + t_2')$. Forming the set $C$ in step $2$ then avoids redundancies so that no points in the set $K$ are inverse to each other.

The algebraic recognition in Steps 4 and 5 follow since the values are the distinct $x$-coordinates of $N$-torsion points.  
With an equation for $E$ and the polynomial representing the kernel of the isogeny $\widehat{\psi}\colon  E \to E'$, we can use V\'elu's formula to calculate $\widehat{\psi}$ explicitly. Taking the dual to $\widehat{\psi}$ gives the desired isogeny $\psi$.
\end{proof}

\begin{remark} If $n_1$ and $m_2$ above are coprime, then $\ker(\hat{\psi}) \cong \Z/m_2\Z \times \Z/n_1\Z \cong \Z/N\Z$ with $N = n_1m_2$ is cyclic. So, in algorithm \ref{alg:genP}, we need only compute the set of values $\{ m_a(x(P)): a \in \Z/N\Z\}$ for a generating point $P$ of $\ker(\hat{\psi})$. Then, we may take those values as the roots of the kernel polynomial $p(x)$ in step $6$ of algorithm \ref{alg: isogeny}. As the computation of the rational maps $m_{a + bj}(x)$ can be costly, this is a useful simplification. If $n_1$ and $m_2$ are not coprime, let $k := \gcd(n_1, m_2)$. Then, we may factor $\psi = [k] \circ \psi'$ where $[k] \colon E(\Delta) \to E(\Delta)$ is the multiplication by $k$ map and $\psi' : E(\Gamma) \to E(\Delta)$ is the isogeny with cyclic kernel obtained as described above replacing $n_1$ with $n_1/k$ and $m_2$ with $m_2/k$.
\end{remark}

\subsection{Descent using automorphisms}
\label{sec:G}

Returning to our master diagram \eqref{eqn:EXDeltaGamma}, we now consider the top map $\beta \colon E(\Gamma) \to X(\Gamma)$ having computed in the previous section an equation for $\psi$ and $E(\Gamma)$ over a number field $K'$.  To do so, we apply a bit of Galois theory.  Associated to our master diagram is the following diagram of inclusions of function fields (see e.g.\ Silverman \cite[\S II.2]{Silverman}):
\begin{equation} \label{eqn:FieldInclusions}
\begin{gathered}
\xymatrix@R=1ex@C=6ex{
\C(E(\Gamma)) \ar@{-}[rd]^{\beta^*} \ar@{-}[dd]_{\psi^*} & \\
& \C(X(\Gamma)) \ar@{-}[dd]^{\varphi^*} \\
\C(E(\Delta)) \ar@{-}[rd]_{\alpha^*} \\
& \C(X(\Delta)) 
} 
\end{gathered}
\end{equation}

We recall our explicit equations from section 3.1 and the automorphisms \eqref{eqn:delta346}.  The inclusion $\alpha^*$ realizes $\C(X(\Delta))$ as the fixed field under $\langle \delta_c^* \rangle$.  For example, for $c=4$ we have 
\[ \C(E_{\square}) = \C(x,y) \]
with $y^2=x^3-x$, and so with $\delta_4(x,y)=(-x,iy)$ we have 
\[ \C(E_{\square})^{\langle \delta_4^* \rangle} = \C(x^2, y^4) = \C(x^2) \subseteq \C(E_{\square}) \]
because $y^4 = x^6 -2x^4 + x^2 \in \C(x^2)$.  

By Lemma \ref{lemma:RGgenByVOMR}, there exists a vertex of maximal rotation (Definition \ref{defn:maxstab}) for $\Gamma$.  At the end of section 2, we argued that up to isomorphism (without loss of generality) we may suppose that this vertex is one of $v_a,v_b,v_c$.  We have $\deg \beta = r(\Gamma) \in \{1,2,3,4,6\}$ equal to the rotation index.  

If $r(\Gamma)=1$, then $E(\Gamma)=X(\Gamma)$ and $\beta$ is the identity.  So we may suppose that $r(\Gamma) > 1$.  

First suppose that $v_c=0$ is a vertex of maximal rotation under a subgroup of rotations generated by a power of $\delta_c$.  Then the quotient map $\beta$ is again by a subgroup of automorphisms of $E(\Gamma)$ over $K'$ as an elliptic curve, so is given in the same well-known manner as in section \ref{sec:F}.  

\begin{lemma} \label{lemma:zeroCoeffs}
Suppose $v_c$ is a vertex of maximal rotation with $r(\Gamma) > 1$.  Then $X(\Gamma) \simeq \PP^1$, and the following statements hold.
\begin{enumalph}
\item If $r(\Gamma)=3,6$, then $E(\Gamma)$ has an equation of the form $y^2=x^3+B$ for some nonzero $B \in K'$, and $\beta \colon E(\Gamma) \to X(\Gamma)$ can be taken to be $(x,y) \mapsto y,y^2$, respectively.  
\item If $r(\Gamma)=4$, then $E(\Gamma) \colon y^2 = x^3 + Ax$ for some nonzero $A \in K'$, and $\beta(x,y)=x^2$.
\item If $r(\Gamma)=2$, then $\beta(x,y)=x$.
\end{enumalph}
\end{lemma}

\begin{proof}
We may suppose that $E(\Gamma)$ has a Weierstrass equation $y^2 = x^3 + Ax + B$.  Any automorphism of $E$ is of the form $(x,y) \mapsto (u^{-2} x, u^{-3} y)$ for some $u \in \C^\times$ with $u^{-4}A=A$ and $u^{-6}B=B$.  Considering the cases $r(\Gamma)=3,4,6$ gives $A=0$ or $B=0$ as in (a) and (b).  We compute the maps in (a)--(c) by considering the fixed subfields under these automorphisms, as above.
\end{proof}

Suppose now that our vertex $v_O$ of maximum rotation is either $v_a$ or $v_b$, with rotations generated by an element $\delta_O$ (generating the coset representatives of $\Gamma/T(\Gamma)$).  In this case, $\delta_O$ need not induce an automorphism of $E(\Gamma)$, because as a rotation of the plane $\delta_O$ need not take the lattice corresponding to $T(\Gamma)$ back to itself.  However, we may simply translate, as in the following lemma.  

\begin{lemma} \label{lemma:autoOnEGp} 
Let $Q_O \colonequals (\wp(v_O),\wp'(v_O)/2) \in E(\Gamma)$ be the image of of $v_O$.  Let $E(\Gamma)'$ denote the elliptic curve whose underlying curve is $E(\Gamma)$ but with origin $Q_O$.  Then we have an isomorphism
\begin{equation}
\begin{aligned}
\tau_{-Q_O} \colon E(\Gamma) &\to E(\Gamma)' \\
P &\mapsto P-Q_O 
\end{aligned}
\end{equation}
of elliptic curves, and $\delta_O$ induces an automorphism of the elliptic curve $E(\Gamma)'$ under $\tau_{-Q_O}$.  
\end{lemma}

\begin{proof}
The translation isomorphism moves $Q_O$ to the origin on $E(\Gamma)'$; thus the action induced by $\delta_O$ is bijective and fixes the origin on $E(\Gamma)'$, so gives an automorphism of $E(\Gamma)'$ as an elliptic curve.
\end{proof}

Thus to compute the map $\beta \colon E(\Gamma) \to X(\Gamma)$, by the lemma we first compose with the isomorphism $\tau_{-Q_O} \colon E(\Gamma) \to E(\Gamma)'$ to reduce to the previous case.  But rather than compute the point $Q_O \in E(\Gamma)$ and the translation map, we find it computationally more convenient to translate by the point $P_O \colonequals (\wp(v_O),\wp'(v_O)/2) \in E(\Delta)$ on the base curve.  

Writing $E(\Delta)'$ for the elliptic curve $E(\Delta)$ having origin $P_O$, we have the following diagram:
\begin{equation}  \label{eqn:EXDeltaGamma_withprime2}
\begin{gathered}
\xymatrix@R=1ex{
& E(\Gamma)' \ar[ddrr]^{\beta'} \ar[ddd]_(.6){\psi} \\ 
E(\Gamma) \ar[rrrd]^(.55){\beta} \ar[ddd]_{\psi} \ar[ur]^{\tau_{-Q_O}} & \\
& & &X(\Gamma) \ar[ddd]^{\varphi}  \\
& E(\Delta)' \ar[ddrr]^{\alpha'}\\
E(\Delta) \ar[rrrd]_{\alpha} \ar[ur]^{\tau_{-P_O}}\\
& & & X(\Delta) 
} 
\end{gathered}
\end{equation}
The diagram is commutative because $\psi(Q_O)=P_O$, both points corresponding to $v_O$ under the complex uniformization.  Note that the map $\psi \colon E(\Gamma)' \to E(\Delta)'$ has the same defining equation as the map $E(\Gamma) \to E(\Delta)$, and still defines a finite map of curves---it just loses the property of being a homomorphism.  

In this way, we have ``aligned" $E(\Delta)'$ with $E(\Gamma)'$, and we can more simply repeat the steps above with $E(\Delta)'$ in place of $E(\Delta)$ at the cost of computing translation maps $\tau_{P_O} \colon E(\Delta) \to E(\Delta)'$ with $P_O$ the image of either $v_a$ or $v_b$, giving a few more fixed maps $\alpha'$, which can be computed by composing $\alpha$ with translation (computed using the group law).

\begin{lemma} \label{lem:alphap346}
The following statements hold, with $E(\Delta)'=E(\Delta)$ as in \eqref{Ehexsquare}.
\begin{enumalph}
\item If $c=6$ and $v_O=v_a=v_2$, then we have
\begin{align*}
\alpha' \colon E(\Delta)' &\to \PP^1 \\
(x,y) &\mapsto \frac{(9\zeta_6 - 9)x^2 + 9\zeta_6x + 9}{x^3 + (3\zeta_6 - 3)x^2 - 3\zeta_6x + 1} 
\end{align*}
\item If $c=6$ and $v_O=v_b=v_3$, then we have
\begin{align*}
\alpha' \colon E(\Delta)' &\to \PP^1 \\
(x,y) &\mapsto \frac{x^6+8x^3y+8x^3+16y^2+32y + 16}{x^6}
\end{align*}
\item If $c=4$ and $v_O=v_a=v_2$, then we have
\begin{align*}
\alpha' \colon E(\Delta)' &\to \PP^1 \\
(x,y) &\mapsto \frac{(x+1)^2}{(x-1)^2}
\end{align*}
\end{enumalph}
\end{lemma}

In case (a), we may need to extend the field of definition $K'$ to include $\zeta_6$.  In the remaining cases, we have taken $v_O=v_c$ without loss of generality, so the maps \eqref{eqn:alpha346} may be used.  After having made this reduction, we drop the superscripts (the underlying curves have the same equations) and proceed to the final step.

\subsection{The Belyi map}
\label{subsection: BelyiMap}

With three of the four maps in our master diagram determined, we complete the computation of $\varphi \colon X(\Gamma) \to X(\Delta)$ by filling in the map in the master diagram from the other three sides, using commutativity.  To do this, we again apply Galois theory, referring to the field diagram \eqref{eqn:FieldInclusions}.  

Let $\xi \colonequals \alpha \circ \psi \colon E(\Gamma) \to X(\Delta)$, a map represented by a rational function $\xi(x,y) \in K'(x,y)$ where $E'=E(\Gamma) \colon y^2 = f'(x)$ is the defining equation of $E'$.  
By commutativity, we have $\xi=\varphi \circ \beta$.  If $r(\Gamma)=1$, then $\beta$ is the identity map so $\varphi=\xi$.  So we may suppose that $r(\Gamma) > 1$.  

The monomial map $\beta \colon E(\Gamma) \to X(\Gamma)$ is described by Lemma \ref{lemma:zeroCoeffs}, corresponding to the cyclic field extension $\C(E(\Gamma))\supseteq \C(X(\Gamma))$, given explicitly by $\beta(x,y)=y^2,x^2,y,x$.  In particular, $\varphi \in \C(X(\Gamma))$ lies in this fixed field, and we need to solve
\[ \xi(x,y) = \varphi(\beta(x,y)) \]
given $\xi$ and $\beta$ explicitly for $\varphi$.  Accordingly, we can write $\xi(x,y)$ as a rational function in the monomial $\beta(x,y)$, using the relation $y^2=f'(x)$ if necessary, replacing every instance of $\beta(x,y)$ in $\xi(x,y)$ with a new variable $u$.  Then $\varphi(u) \in K'(u)$ defines the map $\varphi \colon X(\Gamma) \simeq \PP^1 \to \PP^1$.  

\begin{remark}
We have seen that Euclidean Belyi maps can be understood as descending an isogeny along a fixed quotient map; this is encoded in our master diagram.  Our effort has been to take as input a permutation triple and then to compute the master diagram (associated isogeny and then its descent).  One can also cut this in the middle, working directly with the master diagram by specifying a pair $(K,H)$ where $K \leq \Z[j]/N\Z[j] \simeq (\Z/N\Z)^2$ is a subgroup containing an element of order $N$ and $H \leq \langle j \rangle$ is a subgroup with $H \neq \{\pm 1\}$ and $HK = K$.  This data defines an isogeny to $E(\Delta)$ dual to the one provided by the torsion subgroup, and the descent is along the subgroup of automorphisms, with $H$ stabilizing this kernel.  
\end{remark}

\subsection{Proof of main result}
\label{sec:ProofOfMainRes}

To finish, we put all of the pieces together.  

\begin{alg} \label{alg:thisisit}
This algorithm takes as input a Euclidean, transitive permutation triple $\sigma = (\sigma_a,\sigma_b, \sigma_c) \in S_d^3$ corresponding to a homomorphism $\pi\colon  \Delta \to S_d$ with $\pi(\delta_s) = \sigma_s$ for $s=a,b,c$; it gives as output a model for the corresponding Belyi map from $X(\Gamma)$ to $\PP^1$.
\begin{enumalg}
\item Apply the preprocessing step by calling Algorithm \textup{\ref{alg:preproc}}, with vertex of maximal rotation $v_O$.  
\item Depending on the case of $(a,b,c)$ and $v_O$, look up $\beta$ using Lemma \textup{\ref{lemma:zeroCoeffs}} and the map $\alpha \colon E(\Delta) \to \PP^1$ using Lemma \textup{\ref{lem:alphap346}} (referring back to \textup{\ref{eqn:alpha346}}).
\item Call Algorithm \textup{\ref{alg: basisForTG}} to compute a basis $\eta_1,\eta_2$ for $T(\Gamma)$ and $N=[T(\Delta):T(\Gamma)]$.  
\item Call Algorithm \textup{\ref{alg:genP}} to compute $\psi \colon E(\Gamma) \to E(\Delta)$.  
\item Compute the composition $\xi \colonequals \alpha \circ \psi$.  
\item From $\xi=\varphi \circ \beta$, compute $\varphi \colon X(\Gamma) \to \PP^1$ by substitution.  Return $\varphi$.  
\end{enumalg}
\end{alg}

\begin{theorem} \label{theo:theo}
Algorithm \textup{\ref{alg:thisisit}} terminates with correct output.
\end{theorem}

\begin{proof}
Correctness follows from our master diagram \eqref{eqn:EXDeltaGamma} and the correctness of each step, provided by the proof of correctness of the algorithm used except for Step 6, which is justified in section \ref{subsection: BelyiMap}.  
\end{proof}

\begin{remark} \label{remark:AlgChanges}
In the above, we assumed throughout a Euclidean triangle group $\Delta$ with three generators $\delta_a, \delta_b,$ and $\delta_c$ with orders $a, b,$ and $c$ respectively and satisfying $\delta_c\delta_b\delta_a= 1$.  These three generators corresponded to rotations around the three vertices of a designated triangle in the corresponding tessellation of the plane. We took as input to our algorithm the set of all permutation triples $\sigma = (\sigma_a, \sigma_b, \sigma_c)$ such that $\pi \colon \Delta \to S_n$ taking $\delta_i$ to $\sigma_i$ described a group homomorphism with transitive image.  In some contexts, we might prefer to work with the relation $\delta_a \delta_b \delta_c = 1$.  The change amounts to a relabeling of vertices so that $v_a, v_b,$ and $v_c$ follow each other counterclockwise around a chosen triangle.  

Accordingly, given a permutation triple $\sigma'$ with $\sigma_a'\sigma_b'\sigma_c'=1$, we just take inverses $\sigma_s \colonequals (\sigma_s')^{-1}$ to obtain $\sigma_c\sigma_b\sigma_a=1$, and we call our algorithm above with this inverted input.  
\end{remark}

\section{Examples and data}
\label{sec: Examples}

We conclude with some examples computed using an implementation of Algorithm \ref{alg:thisisit}.

\subsection{Description of implementation}
\label{sec:imp}

We implemented Algorithm \ref{alg:thisisit} using the \textsf{Magma} computer algebra system \cite{Magma}. In particular, we used the existing implementation of V\'elu's formula in calculating our isogeny $\psi$ and the implementation of division polynomials.  The construction of these isogenies is the most time intensive step in our calculation, as in general it involves working in a number field of possibly large degree. Even with this step, most of our example computations take no more than a few seconds to finish.  An example in degree 100 took only 30 seconds.  

\begin{remark}
Returning to Remark \ref{rmk:homothetylattice}, we see that \textsf{Magma} provides two periods for $E$ that span its associated lattice, so we are careful to generate our basis vectors for $T(\Gamma)$ and to deal with lattice coordinate points relative to the lattice \textsf{Magma} uses in its computations.  As we only need worry about this for our two canonical elliptic curves, we can see which lattice \textsf{Magma} uses, compare it to our own lattices described above, and convert coordinates between the two by a simple change of basis operation.
\end{remark}

\subsection{Belyi maps obtained from triples}
\label{sec:maps}

We give here some examples to illustrate Algorithm \ref{alg:thisisit}.  We list the final Belyi maps from $\varphi: X(\Gamma) \to \PP^1$ and provide factorizations of the numerator, denominator, and their difference in the case of genus zero maps, confirming the correspondence between ramification at $0$, $\infty$, and $1$ respectively and the cycle structure of  $\sigma$.  (We provide monic factorizations, ignoring leading coefficients.)

\begin{example} \label{example:firstmap}
Given the permutation triple $\sigma \colonequals( (2\, 4\, 3),(1\, 3\, 4),(1\, 2\, 3))$, we will illustrate the steps in our algorithm and determine the corresponding Belyi map. First, we call Algorithm \ref{alg:preproc} and conjugate $\sigma$ by the transposition $(1\, 4)$ to obtain $((2\, 1\, 3), (4\, 3\, 1), (4\, 2\, 3))$ where $v_c$ is then the vertex of maximal rotation. Since this conjugate triple gives an isomorphic Belyi map, we will redefine $\sigma \colonequals ((2\, 1\, 3), (4\, 3\, 1), (4\,2\,3))$. Since $\omega_1 \colonequals \delta_b\delta_c^2$ and $\omega_2 \colonequals \delta_b^2\delta_c$ span the translations in $T(\Delta)$ by Corollary \ref{cor:Tgens}, we take $\sigma_1 = \pi(\omega_1) = (1\,3)(2\,4)$ and $\sigma_2 = \pi(\omega_2) = (1\,4)(2\,3)$ and call Algorithm \ref{alg: basisForTG}. 
We find our basis vectors for $T(\Gamma)$ are $\eta_1 = \omega_1^2$ and $\eta_2 = \omega_2^2$ so $n_1 = 2, n_2 =0, m_1 =0,$ and $m_2 =2$.

We obtain the rotation index 
\[r = \frac{cn_1m_2}{d} = \frac{3(2)(2)}{4} = 3 \]
and take $N = [T(\Delta):T(\Gamma)] = n_1 m_2 = 4$, so the points in $T(\Delta) \backslash \C$ in the kernel of the multiplication by $N$ map from $T(\Delta) \backslash \C$ to $T(\Gamma) \backslash \C$ are 
\[A = \{(0,0),(1/2, 0), (0, 1/2), (1/2, 1/2)\} \]
with coordinates relative to $\omega_1$ and $\omega_2$, while the points whose images on $E(\Delta)$ have distinct $x$-coordinates are $K:=\{ (1/2, 0), (0,1/2), (1/2,1/2) \}$ as in Step 3 of Algorithm \ref{alg: isogeny} . Letting $k_1, k_2,$ and $k_3$ be the $x$-coordinates of the images of these three points on $E(\Delta)$, we obtain the kernel polynomial
\[p(x) = (x- k_1)(x-k_2)(x-k_3) = x^3 + 1\]
which we input to V\'elu's formula and take the dual to obtain the isogeny $\psi\colon E(\Gamma) \to E(\Delta)$ given by
\[ \psi(x,y) =     \left(\frac{(1/16)x^4 - 32x}{x^3 + 64},\frac{(1/64)x^6y + 20x^3y - 512y}{x^6 + 128x^3 + 4096} \right)\]
and see that $E(\Gamma)$ is given by the equation $y^2 = x^3 +64$.  

Since we are in the $\Delta(3,3,3)$ case, our map $\alpha \colon E(\Delta) \to \PoneC$ is given by $\alpha(x,y) = (y+1)/2$, so the composition $\xi = \alpha \circ \psi \colon E(\Gamma) \to \PoneC$ is given by
\begin{align*}
\xi(x,y) &= \alpha\left(\frac{(1/16)x^4 - 32x}{x^3 + 64}, \frac{(1/64)x^6y + 20x^3y - 512y}{x^6 + 128x^3 + 4096}\right) \\
         &= \frac{(1/128)x^6y + (1/2)x^6 + 10x^3y + 64x^3 - 256y + 2048}{x^6 + 128x^3 +4096}
\end{align*}

Finally, since $r=6$, the map $\beta \colon E(\Gamma) \to X(\Gamma)$ has $\beta(x,y)=y$.  So, we wish to rewrite $\xi$ in terms of only $y$. Since points on $E(\Gamma)$ satisfy $x^3 = y^2 -64$, we may replace each instance of $x^3$ in $\xi$ with $y^2 -64$; we obtain a rational function in $y$, which gives our final Belyi map
\[ \varphi(x) = \frac{(1/128)x^4 + (1/2)x^3 + 9x^2 - 864}{x^3}\]
Let $N(x)$ and $D(x)$ be the numerator and denominator of $\varphi$ respectively. Note that the preimages under $\varphi$ of $0, \infty$, and $1$ respectively are the roots of  $N, D$, and $N - D$. To confirm the ramification of $\varphi$, we note that up to a constant multiple we have the factorizations
\begin{align*}
N(x) &= (x-8)(x+24)^3 \\
D(x) &= x^3 \\
N(x) - D(x) &= (x +8)(x-24)^3
\end{align*}
where the repeated factors confirm the ramification, and we note the direct correspondence between the powers of the factors and the cycle structure of $\sigma$.
\end{example}

\begin{example}
Given $\sigma \colonequals ((1\, 4)(2\, 5)(3\, 6),(1\, 3\, 5),(1\, 4\, 5\, 2\, 3\, 6))$, we determine that $X(\Gamma)$ has genus $0$ and the corresponding Belyi map $\varphi \colon X(\Gamma) \to \PoneC$ is given by
\[
\varphi(x) =\frac{ x^6 + 162x^5 + 7047x^4 + 43740x^3 + 413343x^2 + 1062882x + 4782969}{x^6 -54x^5 + 1215x^4 - 14580x^3 + 98415x^2 - 354294x + 531441}
\]
with numerator, denominator, and difference given by 
\begin{align*}
N(x) &= (x^3 + 81x^2 + 243x + 2187)^2\\
D(x) &= (x-9)^6 \\
N(x) - D(x) &=  (x^2 + 27)(x+9)^3    
\end{align*}
\end{example}

\begin{example}
Given the triple $\sigma \colonequals ((1\, 9)(2\, 8)(3\, 7)(4\, 6), (1\, 6)(2\, 9\, 10\, 3)(4\, 5\, 8\, 7),\\ (1\, 2\, 5\, 4)(3\, 8)(6\, 7\, 10\, 9))$ we obtain
\[ 
\varphi(x) =\frac{\splitfrac{1/625x^{10} + 1/125(8i + 44)x^8 + 1/25(264i + 702)x^6} {+1/5(2872i + 4796)x^4 + (10296i + 11753)x^2}}{\splitfrac{x^8 +1/5(152i - 164)x^6 + 1/25(-18696i + 1422)x^4} {+1/125(547048i + 434764)x^2 + 1/625(-1476984i - 9653287)}}
\]
with numerator, denominator, and difference respectively given by 
\begin{align*}
&x^2 (x^2 + 10i + 55)^4 \\
&(x^2 + 1/5(38i - 41))^4\\
&(x - 4i + 3)(x + 4i - 3)(x^2 + (-2i + 14)x - 24i - 7)^2 (x^2 + (2i - 14)x - 24i - 7)^2
\end{align*}
where $i^2=-1$.
\end{example}

\begin{remark}
Unfortunately, our algorithms do not automatically descend the Belyi map to a minimal field of definition (if such a field exists).  For example, for the permutation triple $\sigma \colonequals ((1\,4),(1\,2\,6)(3\,4\,5),(1\,6\,2\,4\,3\,5)$ we find the map
\[ \varphi(x) = 36(\zeta_6-1)\frac{(x-2)(x-2\zeta-1)^2(x^2+2x-11)}{(x+2z-3)^6} \]
defined over $\Q(\zeta_6)$; however, it can be shown that the Belyi map descends to $\Q$, given more simply by $\varphi(x) = 9(3x^6-3x^4+x^2)$.  We refer to Sijsling--Voight \cite[\S 6]{SijslingVoight} and Musty--Schiavone--Sijsling--Voight \cite[\S 4]{MSSV} for further discussion.
\end{remark}

\end{document}